\documentclass[preprint]{ejpecp} % replace ECP by EJP if needed.
\usepackage{bbm}
\usepackage{enumerate}
\usepackage{xcolor}
\usepackage[normalem]{ulem}
\usepackage{xfrac}
\SHORTTITLE{Tagged particles and size-biased dynamics in mean-field IPS}

\TITLE{Tagged particles and size-biased dynamics in mean-field interacting particle systems
  } % \thanks is optional. Insert line breaks with \\

 \AUTHORS{%
  Angeliki Koutsimpela\footnote{University of Augsburg. \EMAIL{angeliki.koutsimpela@uni-a.de, stefan.grosskinsky@uni-a.de}}
  \and
  Stefan Grosskinsky\footnotemark[1]}%AUTHORS
\KEYWORDS{interacting particle system ; tagged particle ; size-biased empirical process ; mean-field scaling limit} % Separate items with ;

\AMSSUBJ{NA} % Edit. Separate items with ;
\SUBMITTED{January 2, 2013} % Edit.
\ACCEPTED{December 13, 2014} % Edit.

\VOLUME{0}
\YEAR{2023}
\PAPERNUM{0}
\DOI{10.1214/YY-TN}

\ABSTRACT{We establish a connection between tagged particles and size-biased empirical processes in interacting particle systems, in analogy to classical results on the propagation of chaos. In a mean-field scaling limit, the evolution of the occupation number on the tagged particle site converges to a time-inhomogeneous Markov process with non-linear master equation given by the law of large numbers of size-biased empirical measures. The latter are important in recent efforts to understand the dynamics of condensation in interacting particle systems.
}

\newcommand{\W}{W}

\newcommand{\R}{\mathbb{R}}
\newcommand{\E}{\mathbb{E}}

\renewcommand{\P}{\mathbb{P}}

\newcommand{\Lcal}{\mathcal{L}}
\newcommand{\N}{\mathbb{N}}

\begin{document}

%%%%%%%%%%%%%%%%%%%%%%%%%%%%%%%%%%%%%%%%%%%%%%%%%%%%%%%%%%%%%%%%%%%
%%                                                               %%
%% \Wo need for \maketitle.                                       %%
%%                                                               %%
%%%%%%%%%%%%%%%%%%%%%%%%%%%%%%%%%%%%%%%%%%%%%%%%%%%%%%%%%%%%%%%%%%%

%%%%%%%%%%%%%%%%%%%%%%%%%%%%%%%%%%%%%%%%%%%%%%%%%%%%%%%%%%%%%%%%%%%
%%                                                               %%
%% Please replace what follows by the body of your article       %%
%% (up to the bibliography):                                     %%
%%                                                               %%
%%%%%%%%%%%%%%%%%%%%%%%%%%%%%%%%%%%%%%%%%%%%%%%%%%%%%%%%%%%%%%%%%%%

\section{Introduction}

Based on classical results in \cite{sznitman1991topics}, propagation of chaos and laws of large numbers for empirical processes have recently attracted significant attention mostly for mean-field interacting diffusion models (see e.g. \cite{chaintron2022propagation,daipra2017} and references therein). In the context of interacting particle systems (IPS), propagation of chaos has been studied for the evolution of tagged particle locations on regular lattices \cite{rezakhanlou1994evolution,rezakhanlou1994propagation} and for single-site dynamics in mean-field models \cite{grosskinsky2019derivation}, with recent results also for sparse random graphs \cite{ramanan2023interacting}. 
This note is based on results in \cite{grosskinsky2019derivation} which provides a law of large numbers for empirical processes with a connection to rate equations studied in the context of cluster aggregation models \cite{ben2003exchange,schlichting2020exchange}.

We consider the evolution of size-biased empirical measures, which is a useful tool to study the dynamics of condensing IPS with unbounded occupation numbers, such as zero-range \cite{godreche2016coarsening,jatuviriyapornchai2016coarsening} or inclusion processes \cite{chleboun2023sizebiased}. The dynamics of cluster formation in condensing IPS has attracted significant recent research interest \cite{armendariz2023fluid,beltran2015martingaleproblem}, also in the context of metastability (see e.g. \cite{kim2021condensation,landim2023resolvent} and references therein). We show that the occupation number on a tagged particle location in the mean-field limit converges to a time-inhomogeneous Markov process with non-linear master equation given by the law of large numbers for size-biased empirical processes. 
  \textcolor{black}{This provides a new interpretation of the limiting dynamics of size-biased empirical measures, in analogy to classical propagation of chaos \cite{grosskinsky2019derivation,sznitman1991topics} which links the dynamics of unbiased empirical measures with that of occupation numbers on a fixed site. Note also that in contrast to the occupation number, the location of the tagged particle does not converge to a limiting process in the mean-field limit we consider here.} 
Our main assumption is a bound on the jump rates by a bi-linear function of departure and target site occupation, which includes the above mentioned examples of condensing systems. In such models, higher order correlation functions diverge with time, so in contrast to recent results with uniform-in-time estimates \cite{lacker2023sharp} our results can be only local in time.

\section{Notation and main result \label{sec:notation}}

\subsection{Mathematical setting}

We consider stochastic particle systems $({\eta}(t):t>0)$ on finite lattices $\Lambda$ of size $|\Lambda|=L$. Configurations are denoted by ${\eta} =(\eta_x :x\in\Lambda)$ where $\eta_x \in\mathbb{\N}_0$ is the number of particles on site $x$. We consider systems with a fixed number of particles $N=\sum_{x\in\Lambda} \eta_x$ and the state space of all such configurations is denoted by $E_{L,N}\subset \N_0^\Lambda$. 
The dynamics of the process is defined by the infinitesimal generator 
\begin{equation}
	\label{eq:GenMis}
	(\mathcal{L}g)({\eta})=\sum_{x,y\in\Lambda}q(x,y)c(\eta_{x},\eta_{y})(g({\eta}^{x\rightarrow y})-g({\eta}))\ ,\quad g\in C_b (E_{L,N})\ .
\end{equation}
Here, the usual notation ${\eta}^{x\rightarrow y}$ indicates a configuration where one particle has moved from site $x$ to $y$, i.e. $\eta_{z}^{x\rightarrow y}=\eta_{z}-\delta_{z,x} +\delta_{z,y}$, and $\delta$ is the Kronecker delta. 
Since $E_{L,N}$ is finite, the generator \eqref{eq:GenMis} is defined for all bounded, continuous test functions $g\in C_{b}(E_{L,N})$. For a general discussion and the construction of the dynamics on infinite lattices see e.g.\ \cite{andjel2021,balazs2007existence}. 

To ensure that the process is non-degenerate, the jump rates satisfy
\begin{equation}\label{cassum}
	\left\{ \begin{array}{cl}
		c(0,l)=0\;&\mbox{for all }l\geq 0\\
		c(k,l)>0\;&\mbox{for all }k>0\;\mbox{and }l\geq 0.
	\end{array} \right.
\end{equation}
Our main further assumption on the dynamics is that the rates grow sublinearly, in the sense that they are bounded by a bilinear function
\begin{equation}
	\label{eq:lip}
	c(k,l)\leq C k (1+l)\quad\mbox{for constant }C >0\ .
\end{equation}
We focus on complete graph dynamics, i.e. $q(x,y)=1/(L-1)$ for all $x \neq y$, and under the above conditions the process is irreducible on $E_{L,N}$ and
\begin{equation}\label{eq:consmass}
	\sum_{x\in\Lambda} \eta_x (t)\equiv N\quad\mbox{is the only conserved quantity}\ .
\end{equation}

\noindent 
To follow the location $(X(t):t\geq 0)$ of a tagged particle, we extend the state space to $E:=E_{L,N} \times \Lambda$ and states $(\eta ,x)\in E$ describe the particle configuration $\eta\in E_{L,N}$ and location $x\in\Lambda$ of the tagged particle. 
% After selecting a particular particle, in order to study the evolution of the tagged particle system, we have to extend the state space $E_0$ to , i.e. by giving the number of particles at each site, $\eta \in \N_0^{\Lambda}$, and the location of the tagged particle, $x\in \Lambda$.   
In the following, we denote by $\mathbb{P}^L$ and $\mathbb{E}^L$ the law and expectation on the path space $\Omega=D_{[0,\infty)}(E)$ of the joint process $\big( (\eta(t),X(t)): \;t\geq 0\big)$. 
As usual, we use the Borel $\sigma$-algebra for the discrete product topology on $E$, and the smallest $\sigma$-algebra on $\Omega$ such that $\omega\mapsto(\eta_t(\omega),X_t(\omega))$ is measurable for all $t\geq 0$. 	
The joint process is Markov and its evolution is described by the infinitesimal generator
\begin{multline}\label{taggedsyst}
	\tilde{\Lcal} G(\eta,x)= \sum_{y,z\in\Lambda}\frac{1}{L-1} c(\eta_{y},\eta_{z})(G({\eta}^{y\rightarrow z},x)-G({\eta},x))(1-\delta_{xy}) \\
	+ \sum_{z\in\Lambda} \frac{1}{L-1}c(\eta_x,\eta_z) \left[ \frac{\eta_x-1}{\eta_x}\left( G(\eta^{x\rightarrow z},x)-G(\eta,x)\right)+ \frac{1}{\eta_x}\left( G(\eta^{x\rightarrow z},z)-G(\eta,x) \right) \right]
\end{multline}
for all bounded continuous functions $G\in C_b (E)$. 
We consider the empirical processes $t\mapsto F_k^L ({\eta}(t))$ with
\begin{equation}
	F_{k}^L ({\eta}):=\frac{1}{L}\sum_{x\in\Lambda}\delta_{\eta_{x},k} \in [0,1]\ ,\quad k\geq 0\ , \label{fk}
\end{equation}
counting the fraction of lattice sites for each occupation number $k\geq 0$.

For our main result we will consider the thermodynamic limit with density $\rho$, i.e.
\begin{equation}\label{thermo}
	L\to\infty ,\ N=N_L \to\infty \quad\mbox{such that}\quad N/L\to\rho\geq 0\ .
\end{equation} 
Under condition \eqref{thermo}, the sequence $N/L$ is bounded from above by a constant and for simplicity and without loss of generality, we assume that
% this constant equals the limit density $\rho$, i.e.
\begin{equation}\label{thermob}
	 N/L\leq \rho\quad\mbox{for all }L\geq 2\ .
\end{equation} 
For the sequence (in $L$) of initial conditions $(\eta(0) ,X(0))$ we first require the minimal condition that there exists a fixed probability distribution $f (0)$ on $\N_0$ with finite moments
\begin{equation}\label{initialcon0}
	m_1(0) := \sum_{k} kf_k(0) =\rho < \infty  \quad\text{and}\quad   m_2(0) := \sum_{k\geq 1}k^2 f_k(0) < \infty,
\end{equation} 
such that we have a weak law of large numbers
\begin{equation}\label{initialcon0b}
	F_k^L (\eta(0)) \stackrel{d}{\longrightarrow} f_k (0) \quad\text{as }L\to\infty ,\ \text{for all } k \geq 0.
\end{equation}

We need further regularity assumptions on the initial conditions, namely a uniform bound of second and third moments, for some fixed $\alpha_2, \alpha_3 >0$
\begin{align}
	\label{initialcon0c}
	%\eta (0)&\in\Omega_\alpha :=\Big\{\eta :\frac{1}{L}\sum_{x \in \Lambda}\eta_x (0)\leq \alpha_1 \Big\}\quad\mbox{and}\nonumber\\
	%m_2^L (0)&\leq \alpha_2 \quad\mbox{for all }L\geq 2\ ,
	 \E \Big[\frac{1}{L}\sum_{x \in \Lambda}\eta_x^2(0) \Big]\leq \alpha_2 \quad\mbox{and}\quad \E \Big[\frac{1}{L}\sum_{x \in \Lambda}\eta_x^3(0)\Big]\leq \alpha_3 \quad\mbox{for all }L\geq 2\ .%\subset E_{L,N}\quad\mbox{for all }L\geq 2\ ,
\end{align}
% for some fixed $\alpha_2, \alpha_3 >0$.
% and write $E^\alpha =E_{L,\W}^\alpha \times \Lambda \subset E^\alpha$. 
Note that (\ref{thermob}) and conservation of mass \eqref{eq:consmass} imply for the first moment that
\begin{equation}
	\label{alpha}
	%\sum_{k\geq 1} k\, f_k^L (t)
	\frac{1}{L}\sum_{x\in\Lambda} \eta_x (t)=\sum_{k\geq 0}k F_k^L (\eta (t))=\frac{N}{L}\leq \rho \ , \quad\P^L -a.s.\ \mbox{for all }t\geq 0\mbox{ and }L\geq 2\ .
\end{equation}
% and this part of the assumption is no restriction. 
We assume that $N-1$ particles are distributed on the lattice according to some initial conditions satisfying \eqref{initialcon0}, \eqref{initialcon0b}, \eqref{initialcon0c} and the $N$-th particle (the tagged one) is located on position $X(0)$, increasing the value of $\eta_{X(0)} (0)$ by $1$ such that
\begin{equation}\label{initialcon0d}
	\E^L \left[ \eta_{X(0)}^2(0) \right]<   {\alpha}_4 \quad\text{holds for some fixed }\alpha_4 >0 \text{ and all }L\geq 2\ .
\end{equation}

For example, if we distribute $N-1$ particles uniformly, independently on $\Lambda$, \eqref{initialcon0}, \eqref{initialcon0b} are satisfied with Poisson distribution $f(0)$, and  condition \eqref{initialcon0c} is satisfied for all $L\geq 2$. 
There are various ways to then choose the initial position of the tagged particle such that 
% the conditions \eqref{initialcon0}, \eqref{initialcon0b}, \eqref{initialcon0c} and 
\eqref{initialcon0d} is satisfied. We could pick a fixed site (e.g. $X(0)=1$) or select one uniformly at random. On the other hand, selecting a site with the maximum occupation number would lead to logarithmic growth with respect to $L$ of $\eta_{X(0)}(0)$, violating \eqref{initialcon0d}.
% , i.e. $\eta_{x(0)}(0)=o(\sqrt{L}).$ 

\subsection{A law of large numbers for empirical processes}

A law of large numbers for the empirical process \eqref{fk} was established in \cite{grosskinsky2019derivation}. 
% To formulate it, we use the notation
% \[
% \langle F^L (\eta ),h\rangle :=\sum_{k\geq 0} F_k^L (\eta )\, h(k)
% \]
% to test the empirical process against a function $h:\N_0 \to\R$.

\begin{theorem}
	\label{thmfactorization}
	Consider a process with generator \eqref{eq:GenMis} on the complete graph with sublinear rates (\ref{eq:lip}) and initial conditions satisfying \eqref{initialcon0}, \eqref{initialcon0b} and the second moment condition in \eqref{initialcon0c}. 
	Then we have in the thermodynamic limit \eqref{thermo} for any $\rho >0$ and any Lipschitz function $h:\N_0\to\R$,
	\begin{equation}
		\Big(\sum_{k\geq 0} F_k^L (\eta(t) )\, h(k) :t \geq0\Big)\to \Big(\sum_{k\geq 0} f_k (t)\, h(k) :t\geq0\Big)\quad\mbox{weakly on }    {D_{[0,\infty)}(E)},
	\end{equation}
	% \begin{equation}
		% \big(\langle F^L (\eta (t)),h\rangle :t\geq 0\big)\to \big(\langle f(t),h\rangle :t\geq 0\big)\quad\mbox{weakly on path space as }L\to\infty\ ,
		% \end{equation}
	where 
	$f(t)=(f_k(t):k\in \mathbb{N}_0)$ is the unique global solution to the \textbf{mean-field equation}
	\begin{align}
		\label{eq:mis_diff_f_k}
		\frac{df_{k}(t)}{dt}
		&=\sum_{l\geq 0}c(k+1,l)f_l(t)f_{k+1}(t)+\sum_{l\geq 1} c(l,k-1) f_l(t)f_{k-1}(t) \nonumber \\
		&\quad-\bigg(\sum_{l\geq 0}c(k,l)f_l(t)+\sum_{l\geq 0} c(l,k)f_l(t)\bigg)f_{k}(t)\quad\mbox{for all }k\geq 0,
	\end{align}
	with initial condition $f(0)$ given by (\ref{initialcon0b}). 
	Here we use the convention $f_{-1}(t)\equiv 0$ for all $t\geq 0$ and recall that $c(0,l)=0$ for all $l\geq 0$.
\end{theorem}
Notice that in \cite{grosskinsky2019derivation}, this result was established for bounded functions $h:\N_0\to\R$ and more restrictive assumptions on initial conditions. The proof for Lipschitz functions $h:\N_0\to\R$ as stated above is included in Appendix \ref{app:a}.

The nonlinear equations \eqref{eq:mis_diff_f_k} can be written as
\[
\frac{df_{k}(t)}{dt} =\mu_{k+1} (t)\, f_{k+1}(t)+\beta_{k-1} (t)\, f_{k-1}(t) -\Big(  {\mu_k (t)+\beta_k (t)}\Big) f_{k}(t)\ ,\quad k\geq 0,
\]
and thus be identified as the master equation of a non-linear birth-death chain on $\N_0$ with time-dependent birth and death rate
\begin{equation}\label{bdrate}
	    {\mu_k (t)  {:=}\sum_{l\geq 0} c(k,l) f_l (t) \quad\text{and}\quad \beta_k (t)  {:=}\sum_{l \geq 1} c(l,k)f_l (t)}\, ,
\end{equation}
respectively. Here, we use again the convention $\beta_{-1} (t)\equiv \mu_0 (t)\equiv 0$. This corresponds to the limiting dynamics of the occupation number of a fixed site, where any finite set of those evolves as independent birth-death chains according to the propagation of chaos (see \cite{grosskinsky2019derivation} and references therein for details).

The solutions $(f_k (t):k\geq 0)$ to this system of equations have been studied in \cite{grosskinsky2019derivation,jatuviriyapornchai2016coarsening} and in detail in \cite{schlichting2024variational,schlichting2020exchange}. In condensing systems, solutions show a bump at occupation numbers increasing with time corresponding to the emergence of cluster sites in the condensed phase. The volume fraction of the latter vanishes in time and corresponds to the integral of the bump. To study the asymptotics of the condensed phase, it is therefore advantageous to consider a size-biased empirical distribution, as has been done for zero-range \cite{jatuviriyapornchai2016coarsening} and inclusion processes \cite{chleboun2023sizebiased,jatuviriyapornchai2020structure}. Since \eqref{eq:mis_diff_f_k} conserves the total mass $\rho \equiv m_1 (t)=\sum_{k\geq 1} kf_k (t)$ for all $t\geq 0$, the corresponding size-biased quantities
\begin{equation}\label{eq:pkdef}
	p_k (t):=\frac{1}{\rho} kf_k (t)\ ,\quad k\geq 1\quad\mbox{are normalized with}\quad \sum_{k\geq 1} p_k (t)\equiv 1\ ,
\end{equation}
and describe the fraction of mass in clusters of size $k$. 
From \eqref{eq:mis_diff_f_k} and \eqref{eq:pkdef} it is easy to see that they solve
\begin{align}\label{sbmfe}
	\frac{dp_{k}(t)}{dt} &=\frac{k}{k+1}\mu_{k+1} (t)\, p_{k+1}(t)+\frac{k}{k-1}\beta_{k-1} (t)\, p_{k-1}(t) -\Big(  {\mu_k (t)+\beta_k (t)}\Big)p_{k}(t)\nonumber\\
   &   =   \frac{k}{k+1}\mu_{k+1}(t) p_{k+1}(t) +\beta_{k-1}(t)p_{k-1}(t) +\underbrace{\sum_{n\geq 1}\frac{1}{n} c(n,k-1) f_{k-1}(t)p_{n}(t)}_{=\frac{1}{k-1}\beta_{k-1} (t) p_{k-1} (t)}\nonumber\\
    &\qquad  {-\Big( \frac{k-1}{k}\mu_k(t)+\frac{1}{k}\underbrace{\sum_{n\geq 1} c(k,n-1) f_{n-1}(t)}_{=\mu_k (t)}+ \beta_k(t)\Big)\, p_k(t)}\ ,\qquad k\geq 2\nonumber\\
	% \frac{dp_1(t)}{dt} &=\frac{1}{2}\mu_{2} (t)\, p_{2}(t)+ \underbrace{\frac{1}{\rho}\beta_0 (t)\, f_0 (t)}_{=\sum_{k\geq 1}\frac{c(k,0)}{k} f_{0}(t) q_{k}(t)}    -\Big(  {\mu_1 (t)+\beta_1 (t)}\Big) p_{1}(t) 
% \end{align}
% while for $k=1$
% \begin{align}
    \frac{dp_1(t)}{dt} &=\frac{1}{2}\mu_{2} (t)\, p_{2}(t)+\frac{1}{\rho}\beta_0 (t)\, f_0 (t) -\Big(  {\mu_1 (t)+\beta_1 (t)}\Big) p_{1}(t) \nonumber\\
    &=   {\frac{1}{2}\mu_{2} (t)\, p_{2}(t)+ \sum_{n\geq 1}\frac{c(n,0)}{n} f_{0}(t) p_{n}(t)    -\Big(\mu_1 (t)+\beta_1 (t)\Big) p_{1}(t) }
\end{align}
% and\quad $\displaystyle\frac{dp_1(t)}{dt} =\frac{1}{2}\mu_{2} (t)\, p_{2}(t)+\frac{1}{\rho}\beta_0 (t)\, f_0 (t) -\Big(\beta_1 (t)+\mu_1 (t)\Big) p_{1}(t)$\ ,\\
with initial condition $p_k (0) =kf_k (0)/\rho$, $k\geq 1$.  Based on Theorem \ref{thmfactorization}, one can show that the empirical mass processes
\[
t\mapsto P_k^L (\eta (t)):=\frac{1}{N} \sum_{x\in\Lambda} k\delta_{\eta_x (t),k}=\frac{L}{N} kF^L_k(\eta(t)) \in [0,1]\ ,\quad k\geq 1
\]
converge to solutions of \eqref{sbmfe}. $P_k^L(\eta)$ counts the fraction of particles on sites with $k$ particles.   Following our main result, we will see that \eqref{sbmfe} can be interpreted as the master equation for a process on $\N$, describing the mass on the site of a tagged particle.

\subsection{Main result}

The evolution of the occupation number on the tagged particle site is denoted by $\W^L(t):=\eta_{X(t)}(t)$. To study its dynamics we apply the generator \eqref{taggedsyst} to a test function $G(\eta,x)=g(\eta_x)$
and find
\begin{multline}
	\hat{\Lcal}^L_{\eta} g(\eta_x)= \sum_{y\in\Lambda}\frac{1}{L-1} c(\eta_{y},\eta_{x})(g(\eta_x+1)-g(\eta_x))(1-\delta_{xy}) \\+ \sum_{y\in\Lambda} \frac{1}{L-1}c(\eta_x,\eta_y) \left[\frac{\eta_x-1}{\eta_x}\left( g(\eta_x-1)-g(\eta_x)\right)+  \frac{1}{\eta_x}\left( g(\eta_y+1)-g(\eta_x) \right) \right](1-\delta_{xy}) \ .
\end{multline}
Plugging in the process, this can be written for each $n\geq 1$ as
% \textcolor{red}{LATER n IS THE MOMENT POWER, CAN WE USE l HERE??}
\begin{multline}\label{genN}
	\hat{\Lcal}^L_{\eta(t)}  g(n)=\frac{L}{L-1}\sum_{k \geq 1} c(k,n)F^L_{k}(\eta(t)) \big(g(n+1)-g(n)\big)\\
   {+}\frac{L}{L{-}1}\bigg(\frac{n{-}1}{n} \sum_{k\geq 0} c(n,k)F^L_{k}(\eta(t))\big( g(n{-}1)-g(n)\big)+\frac{1}{n}\sum_{k\geq 0} c(n,k) F^L_{k}(\eta(t)) \left( g(k{+}1)-g(n) \right)  \bigg)\\
 -\frac{1}{L-1}c(n,n)\bigg(\frac{n+1}{n} \left( g(n{+}1)-g(n) \right)+ \frac{n-1}{n} \big( g(n{-}1)-g(n)\big)\bigg)  \ .
\end{multline}
Note that the process $(\W^L(t), t\geq 0)$ is itself not a Markov process, since its generator depends also on the state of the configuration $\eta (t)$. 
Based on Theorem \ref{thmfactorization}, we have that for each $n\in \N$, in the limit $L\to \infty$, \eqref{genN} converges to a time-inhomogeneous generator
\begin{align}\label{limgen}
	\hat{\Lcal}_t g(n)=\beta_n(t)\big(g(n{+}1)-g(n)\big)&+ \frac{n{-}1}{n} \mu_n(t) \big( g(n{-}1)-g(n)\big)\nonumber \\&\quad+ \frac{1}{n}\sum_{k\geq 1} c(n,k{-}1) f_{k-1}(t) \left( g(k){-}g(n) \right) \ .
\end{align}
%with birth and death rate $\beta_n(t)$ and $\mu_n(t)$ . 
This generator describes a birth-death process with time-dependent birth and death rates $\beta_n (t)$ and $\frac{n-1}{n}\mu_n (t)$ as given in \eqref{bdrate}, and with additional long-range jumps when the tagged particle changes position.   {Notice that the master equation that corresponds to this process coincides with \eqref{sbmfe}.} Here is our main result.

\begin{theorem}\label{mainthm}
	Consider a tagged particle process with generator \eqref{taggedsyst} on the complete graph with sublinear rates (\ref{eq:lip}) and initial conditions satisfying \eqref{initialcon0}, \eqref{initialcon0b}, \eqref{initialcon0c} and \eqref{initialcon0d}. In the thermodynamic limit \eqref{thermo}, for any $\rho >0$, 
	\[
	\big( \W^L(t) : t   \geq 0\big) \to \big( \hat{\W}(t): t   \geq 0\big) \quad\text{weakly on }   D_{[0,\infty)}(E) ,
	\]
	where $\big( \hat{\W}(t): t\geq 0\big)$ is a time-inhomogeneous Markov process on $\N$ with generator $\hat{\Lcal}_t$ \eqref{limgen} and corresponding master equation \eqref{sbmfe}.
\end{theorem}

Therefore, in a mean-field scaling limit, the evolution of the occupation number on the tagged particle site, $\eta_{X(t)} (t)$, converges to a time-inhomogeneous  Markov process on $\N$ with (non-linear) master equation \eqref{sbmfe} given by the law of large numbers of size-biased empirical measures.   \textcolor{black}{This provides a direct interpretation of the dynamics of these measures in terms of the underlying particle system in analogy to propagation of chaos for unbiased empirical measures \cite{grosskinsky2019derivation}. Our method of proof also directly extends to occupation numbers on any finite number of tagged particle sites. Even if correlated by initial conditions, they will evolve independently eventually, since tagged particles do not revisit the same site asymptotically in a mean-field scaling limit.} As was demonstrated in \cite{godreche2016coarsening,jatuviriyapornchai2016coarsening} for the example of a condensing zero-range process, this can be used to devise efficient numerical schemes to study the coarsening dynamics of the condensed phase emerging from a supercritical homogeneous initial condition. In particular, the expectation
\[
\E [\hat{\W}(t)] =\sum_{k\geq 1} kp_k (t)=\frac{1}{\rho}\sum_{k\geq 1} k^2 f_k (t)
\]
describes the second moment of the particle system, which is increasing with $t$ following a coarsening scaling law for condensing systems (see e.g.\ \cite{godreche2016coarsening,jatuviriyapornchai2016coarsening,schlichting2020exchange} for details).

% If we tag several particles, their dynamics is obviously correlated while they are on the same site. Since in our limit 
% Comment about independence of the dynamics of several tagged particles...

\section{Proof of the main result\label{sec:proof}}

\subsection{Moment bounds}

As a first step, we collect some useful results on moments and establish a time-dependent bound on the moments of the processes  $\eta_x(t)$ for $x\in\Lambda$ and $\W^L(t)$. 
For any integer $n\geq 0$ denote the $n$-th moment by
\begin{equation}\label{lmoment}
	m_n^L (t):=\E^L \Big[\frac{1}{L}\sum_{x\in\Lambda} \big(\eta_x (t)\big)^n\Big] =\E^L \Big[\sum_{k\geq 0} k^n F_k^L (\eta (t))\Big]\ .
	% =\sum_{k\geq 0} k^n f_k^L (t)\ .
\end{equation}
We have $m_0^L (t)\equiv 1$ and with \eqref{initialcon0b}, $m_1^L (0)\to\rho$ and $m_2^L (0)\to m_2 (0)<\infty$. The uniform conditions \eqref{initialcon0c} on the moments further imply for all $L\geq 2$ that $m_2^L (0)\leq \alpha_2$,   {$m_3^L (0)\leq \alpha_3,$} and with conservation of mass \eqref{alpha}, we have $m_1^L (t)\leq \rho$ for all $t\geq 0$, while higher moments typically grow in time for condensing systems (see e.g. \cite{godreche2016coarsening,jatuviriyapornchai2016coarsening,schlichting2020exchange}). The following result gives a general (but very rough) upper bound.
% as,
% \begin{equation}\label{secmomgrow}
	% m_2^L (t)\leq (\alpha_2 +Ct)e^{Ct}\quad\mbox{for all }t\geq 0\mbox{ and }L\geq 2\ .
	% \end{equation}
% given that they are initially finite, i.e. at $t=0$, according to the following Proposition.

\begin{proposition}\label{gwall}
	Assume that the sequence $\big( m^L_{n}(0)\big)_{L\geq 2}$ is bounded uniformly in $L$ for some integer $n\in\mathbb{N}$. Then there exists a constant $  {B}_n>0$ independent of $L$ such that
	\begin{equation}\label{eq:gwall}
		m^L_n (t) \leq   {B_n} e^{   {B}_n t}\quad\mbox{for all }t\geq 0\mbox{ and }L\geq   {2}\ .
	\end{equation}
\end{proposition}
\begin{proof}
	Applying the generator \eqref{eq:GenMis} to the function $g(\eta)=\eta_x^{  {n}}$ for $n\in \N$ and some $x\in \Lambda$, we get
	\begin{equation}\label{genhm}
		\Lcal \eta_x^{  {n}}=\frac{1}{L-1}\left(\sum_{y\neq x} c(\eta_{  {y}},\eta_{  {x}}) \left((\eta_x+1)^{  {n}}-\eta_x^{  {n}}\right)+ \sum_{y\neq x} c(\eta_{  {x}},\eta_{  {y}}) \left((\eta_x-1)^{  {n}}-\eta_x^{  {n}}\right)\right)   {\ .}    
	\end{equation}
	Note that $  {p^\pm_{n-1} (k):=(k\pm 1)^n -k^n}$ is a polynomial of degree $n-1$, which implies with (\ref{genhm}) and sublinear rates (\ref{eq:lip}) that
	\begin{align}
		\frac{d}{dt}m^L_n(t) &=\frac{1}{L} \sum_{x\in \Lambda} \E^L\left[ \Lcal \eta_x^n(t)  \right]= \frac{1}{L-1} \E^L\bigg[ \sum_{k,l \geq 0}  {c(l,k)}  p^+_{n-1}(k)\Big(F_k^L(\eta(t))L-\delta_{k,l}\Big)F_l(\eta(t))\nonumber\\
		& \quad+ \sum_{k,l\geq 0}   {c(k,l)} p^-_{n-1}(k)\Big(F_k^L(\eta(t))L-\delta_{k,l}\Big) F_l (\eta (t))\bigg]\nonumber\\
		& =\frac{L}{L-1} \E^L\bigg[ \sum_{k,l \geq 0} \Big(   {c(l,k)}p^+_{n-1}(k)+  {c(k,l)}p^-_{n-1}(k) \Big) F_k^L(\eta(t))F_l(\eta(t)) \bigg]\nonumber\\
		& \quad- \frac{1}{L-1} \E^L \bigg[\sum_{k\geq 1}c(k,k) \Big( p^+_{n-1}(k)+p^-_{n-1}(k)\Big) F_k^L(\eta(t))\bigg]\nonumber\\
		&\leq 2C \E^L\bigg[ \sum_{k,l \geq 0} l(1+k)p^+_{n-1}(k) F_k^L(\eta(t))F_l(\eta(t)) \bigg]\nonumber\\
		&\leq  2C\rho\E^L\bigg[ \sum_{k\geq 0}   (1+k)p^+_{n-1}(k)F_k^L(\eta(t))\bigg]
	\end{align}
	Here we used that $p_{n-1}^- (k)\leq 0$ $\forall k\geq 1$ and $p_{n-1}^+ (k)+ p_{n-1}^- (k)\geq 0$ in the first inequality, and conservation of mass \eqref{alpha} in the second inequality.   Since  $m^L_{n} (t)\leq m^L_{n+1} (t)$ for all $n\geq 1$, this implies for some constant $  {\hat{B}}_n >0$
	\[
	\frac{d}{dt} m^L_n (t)\leq    {\hat{B}}_n \big( 1+m^L_n (t)\big)\ ,
	\]
	which implies, based on Gronwall's Lemma and the boundedness of $m^L_n(0)$,  that 
\[
 m^L_n (t)\leq   \big( 1+m^L_n (0)\big)e^{\hat{B}_nt}\leq B_ne^{B_nt} \ .
 \]
for some constant $B_n$ which does not depend on $L$.	
\end{proof}

\noindent
In the following, we denote the  $n$-th moment of the process $\W^L(t)$ by
\begin{equation}\label{hatmom}
	\hat{m}^L_n(t):=\E^L \big[ (\W^L (t))^n \big] =\E^L\big[\left(\eta_{X(t)}(t)\right)^n\big]\ .
\end{equation}
Notice that based on initial condition \eqref{initialcon0d}, we have $\hat{m}^L_2(0)\leq \alpha_4$. Similarly to Proposition \ref{gwall}, we can establish the following (rough) bounds on the moments of this process.

\begin{proposition}\label{gwallΝ}
	Assume that the sequence  $\big( m^L_{n+1}(0)\big)_{L\geq 2}$ is bounded for some integer $n\in \N$. Then, there exists a constant   {$C_n>0$} independent of $L$ such that
	\begin{equation}\label{eq:gwallN}
  	\hat{m}^L_n (t) \leq \big(  { \hat{m}^L_n(0)+C_nt}\big) e^{  {C_n} t}\quad\mbox{for all }t\geq 0\mbox{ and }L\geq   {2}\ .
	\end{equation}
\end{proposition}
\begin{proof}
	Applying the generator \eqref{genN} to the function $g(l)=l^n$ for $n\in \N$, we get 
	% \begin{multline}\label{genhm}
		%   \frac{ d\hat{m}^L_1(t)}{d t}=\E^L \left[\hat{\Lcal}^L_{\eta(t)} \W(t)\right] = \E^L \left(\frac{L}{L-1}\sum_{k \geq 1} c(k,\W(t))F^L_{k}(\eta(t))\right) \\+ \E^L \left( \frac{L}{L-1}\left[\frac{1}{\W(t)}\sum_{k\geq 0} c(\W(t),k) F^L_{k}(\eta(t)) \left(k+1-\W(t)\right)- \frac{\W(t)-1}{\W(t)} \sum_{k\geq 0} c(\W(t),k)F^L_{k}(\eta(t)) \right] \right)\\- \E^L \left(\frac{1}{L-1}c(\W(t),\W(t))\frac{2}{\W(t)}\big[ (\W(t)+1)\left( (\W(t)+1)^k -\W^k(t)\right)-(\W(t)-1)\left( (\W^k(t) -(\W(t)-1)^k\right) \big] \right) \\
		%   \leq C_1(a_1+\hat{m}^L_1(t))+C_1({m}^L_2(t)+C_2a_1)%\leq C_1a_1(C_2+1)+C_1(a_2+Ct)e^{Ct}+C_1\hat{m}^L_1(t)
		% \end{multline}
	% which by Gronwall's Lemma and \eqref{secmomgrow} implies that
	% \begin{equation}
		% \hat{m}^L_1(t)\leq \left[ \hat{m}^L_1(0)+ C_1a_1(C_2+1)t\right]e^{C_1t}+\frac{C_1}{C}(a_2+Ct)e^{(C_1+C)t}
		% \end{equation}
	\begin{multline}\label{genhmn}
		\frac{ d\hat{m}^L_n(t)}{d t}=\E^L \left[\hat{\Lcal}^L_{\eta(t)} (\W^L(t))^n\right] =\frac{L}{L-1} \E^L \left[\sum_{k \geq 1} c(k,\W^L(t))F^L_{k}(\eta(t)) p_{n-1}^+(\W^L(t))\right] \\
		+  \frac{L}{L-1}\E^L \left[\frac{1}{\W^L(t)}\sum_{k\geq 0} c(\W^L(t),k) F^L_{k}(\eta(t)) \Big( (k+1)^n-(\W^L(t))^n\Big)\right]\\
		- \frac{L}{L-1}\E^L \left[ \frac{\W^L(t)-1}{\W^L(t)} \sum_{k\geq 0} c(\W^L(t),k)F^L_{k}(\eta(t))p_{n-1}^+( \W^L(t)-1)) \right] \\
		-\frac{1}{L-1} \E^L \left[\frac{c(\W^L(t),\W^L(t))}{\W^L(t)}\left( (\W^L(t)+1)p_{n-1}^+(\W^L(t))-(\W^L(t)-1)p_{n-1}^+(\W^L(t)-1) \right) \right] 
	\end{multline}
	where we used $p_{n-1}^- (k)=-p_{n-1}^+ (k-1)$. 
	Since the functions $  {l}\mapsto lp_{n-1}^+(  {l})$ are increasing for all $n\in \N$, the last two lines in \eqref{genhmn} are negative and therefore, we have for some polynomial $q_n$  of degree $n$ and a positive constant ${\tilde{C}}_n$
 %$n\mapsto m^{L}_n(t)$ and $n\mapsto \hat{m}^{L}_n(t)$ for all $t\geq 0$ are non-decreasing, we have that for   some polynomial $  q_n$  of degree $n$:
	\begin{align*}
		\frac{ d\hat{m}^L_n(t)}{d t}&\leq \frac{L}{L-1} C\rho\E^L \left[  {q}_n(\W^L(t)) \right] +  \frac{L}{L-1}C\E^L \left[\sum_{k\geq 0} (1+k)^{n+1}F^L_{k}(\eta(t)) \right]\\&
		\leq   {\tilde{C}}_n(  {1+}\hat{m}^L_n(t)+m^L_{n+1}(t))\leq    {\tilde{C}}_n \Big(2\hat{m}^L_n(t) +  B_{n+1} { e^{ B_{n+1}t}} \Big)\ .
	\end{align*}
	In the last line we used relation  \eqref{eq:gwall} and that $n\mapsto m^{L}_n(t)$ and $n\mapsto \hat{m}^{L}_n(t)$ for all $t\geq 0$ are non-decreasing.
 % and $  {\tilde{C}}_n$ is a positive constant. 
 The result then follows by Gronwall's Lemma.
\end{proof}

\noindent
Based on Proposition \ref{gwallΝ} and assumptions \eqref{initialcon0c}, \eqref{initialcon0d}, we have the following corollary.

\begin{corollary}
	   {Under assumptions \eqref{initialcon0c} and \eqref{initialcon0d}}, there exists
  a constant $C_2>0$ independent of $L$ such that
 \begin{equation}\label{sec_mom_boundN}
  \hat{m}^L_2 (t) \leq    {\big(\alpha_4+C_2t\big) e^{C_2 t}}\quad\mbox{for all }t\geq 0\, ,\ L\geq  2\ .
 \end{equation}
 
\end{corollary}

\subsection{Existence of limit processes\label{sec:existence}}

\begin{proposition}\label{tight\W}
	Consider the process with generator \eqref{genN} and conditions as in Theorem \ref{mainthm}. Denote by $\mathbb{Q}^L$ the law of the process $t \mapsto \W^L(t)$ on path space $D_{[0,\infty)}(\N)$, which is the image measure of $ \mathbb{P}^L$ under the mapping $(\eta,x)\mapsto \eta_x$. Then $\mathbb{Q}^L$ is tight as $L \to \infty$.
\end{proposition}
\begin{proof}
To establish tightness for $\mathbb{Q}^L$, we will use a coupling argument. The process $\W^L$ is coupled with a process $\bar{\W}^L$ such that $\bar{\W}^L$ jumps (at least) whenever the process $\W^L$ jumps, with a positive jump of length greater or equal than that of the process $\W^L$. In this way, as demonstrated below, tightness for $\bar{\W}^L$ implies tightness for $\W^L$. 

According to generator \eqref{genN}, for the process $\W^L$ we have: 
\begin{itemize}
    \item Birth rate: The process jumps from $n$ to $n+1$ at rate
    $$ \frac{L}{L-1}\sum_{k \geq 1} c(k,n)F^L_{k}(\eta(t))- \frac{1}{L-1}c(n,n) \leq 2C\rho (1+n)\leq 4C\rho n \ .$$
    \item Death rate: The process jumps from $n$ to $n-1$ at rate
    $$ \frac{L}{L{-}1}\frac{n{-}1}{n} \sum_{k\geq 0} c(n,k)F^L_{k}(\eta(t)) -\frac{1}{L-1}c(n,n) \frac{n-1}{n}\leq 2C(1+\rho)n \ .$$
    \item Long-range jump rate: The process jumps from $n$ to $k+1$ for $k\geq0$  at rate
    $$
    \frac{L}{L{-}1}\frac{1}{n} c(n,k) F^L_{k}(\eta(t))-\frac{1}{L{-}1}\frac{1}{n} c(n,n)\delta_{k,n}\leq 2C(1+k)F^L_{k}(\eta(t))
    \ .$$
\end{itemize}
Based on the above, we consider the jump process $\bar{\W}^L(t)$ as follows:
\begin{itemize}
    \item Birth rate: The process jumps from $n$ to $n+1$ at the increased rate
    $$  \bar{C}n\geq 4C\rho n+2C(1+\rho )n\ ,\quad\mbox{where }\bar{C}:=2C(1+3\rho).$$
    % where $\bar{C}:=2C(1+3\rho).$
    \item Positive long-range jumps: The process jumps from $n$ to $2n+k$ with jump length $n+k>|k+1-n|$ for $k\geq0$ at the increased rate
  $$
        2C(1+k)F^L_{k}(\eta(t)) \ .
        $$
\end{itemize}
Therefore, the generator of the new process $\bar{\W}^L(t)$ is the following
\begin{equation}\label{genN2}
	\bar{\Lcal}^L_{\eta(t)}  g(n)=\bar{C}n \big(g(n+1)-g(n)\big)
 + 2C\sum_{k\geq 0}(1+k)F^L_{k} (\eta(t))  \left( g(2n+k)-g(n) \right) \ .
\end{equation}
Both processes are maximally coupled so that 
% The two processes are constructed and coupled in such way so that 
whenever the process $\W^L(t)$ jumps, the process $\bar{\W}^L(t)$ also jumps with a positive jump of length at least equal to that of the process $\W^L(t)$. Since the rates are monotone increasing with state $n\in\N$, this implies that almost surely under the coupling path measure $\bar\P^L$. 
\begin{equation}\label{comp}
 |\W^L(t+s)-\W^L(t)| \leq \bar{\W}^L(t+s)-\bar{\W}^L(t) \text{ for all } t,s\geq 0 \ .
\end{equation}
Moreover, we start the two processes with the same initial value, i.e.  \begin{equation}\label{init}
    \bar{\W}^L(0)=\W^L(0)\ ,
\end{equation} 
which implies that $\bar \P^L$-almost surely
\begin{equation}\label{comp0}
 1\leq \W^L(t) \leq \bar{\W}^L(t) \text{ for all } t\geq 0 \ .
\end{equation}
Based on \eqref{comp}, \eqref{comp0}, in order to prove tightness for the processes $\{(\W^L(t): t\geq 0)\}_{L\geq 2}$, it suffices to prove tightness for the processes $\{(\bar{\W}^L(t): t\geq 0)\}_{L\geq 2}$. 

Based on Theorem 2.4 and Remark 4.2 in \cite{GUT2001101}, in order to establish tightness for the increasing jump processes $\{(\bar{\W}^L(t): t\geq 0)\}_{L\geq 2}$
% (which only has positive jumps)
, it suffices to prove the following:
\begin{enumerate}[(i)]
    \item  For each $T>0,$\quad 
    $\lim \limits_{a\to \infty}\sup \limits_{L\geq 2}\P( \sup \limits_{0\leq s \leq  T}|\bar{\W}^L(s)|>a)=0   \ .$
    \item For each $0<a_1<a_2$,
    $$
    \delta^{-1}\limsup \limits_{L\to \infty}  \sup \limits_{a_1\leq s \leq a_2} \P ( \text{at least two } \bar{\W}^L-\text{jumps  in } [s,s+\delta)) \to 0 \text{ as } \delta\to 0^+ \ .
    $$
    \item For every $\epsilon>0$,$\quad
    \limsup \limits_{L\to \infty} \P \left(   \bar{\W}^L(t) -\bar{\W}^L(0)>\epsilon \right) \to 0 \text{ as } t\to 0^+ \ . $
    
    % \item For every $\epsilon>0,$
    % $$
    % \limsup \limits_{L\to \infty} \P \left(  \bar{\W}^L(T) -\bar{\W}^L(T-t)>\epsilon \right) \to 0 \text{ as } t\to 0^+ \ . 
    % $$
\end{enumerate}
For simplicity of notation, here and in the following we use the generic notation $\P$ and $\E$ for the law and expectation of the process $\bar W^L$.\\
\textbf{Proof of (i):} Let $T>0$. Since the positive  process $\bar{\W}^L(t)$ is increasing as a function of $t$, it suffices to prove that
$\lim \limits_{a\to \infty}\sup \limits_{L\geq 2}\P( \bar{\W}^L(T)>a)=0   \ .$ 
By Markov's inequality,
\begin{equation}\label{markovbarn}
   \P( \bar{\W}^L(T)>a)\leq  \frac{\E \big[ \bar{\W}^L(T) \big] }{a}\quad\mbox{for all }a>0\ .
\end{equation}
To control the expectation, we establish a bound on the moment 
% consider the following function: 
\begin{equation}
    \bar{m}_n(t)=\E [ (\bar{\W}^L(t))^n ]\ , \quad\text{for all } t\geq 0\ .
\end{equation}
% for which we have the following lemma. 
\begin{lemma}\label{mombarN}
Under assumptions \eqref{initialcon0c} and \eqref{initialcon0d}, there exists
  a constant  {$D_2>0$} independent of $L$ such that
 \begin{equation}\label{sec_mom_bound}
  \bar{m}^L_2 (t) \leq  \alpha_4 e^{e^{D_2 t}}\quad\mbox{for all }t\geq 0\, ,\ L\geq 2\ .
 \end{equation}
\end{lemma}
\begin{proof}
Since $\bar W^L$ is an unbounded process, we will first establish bounds on the moments of the bounded process $\bar W^L\wedge M:=\min (\bar W^L , M)$, namely for 
$$\bar m^L_{2,M}(t)=\E \big[ (\bar W^L(t)\wedge M)^2 \big]\ .$$
Applying  generator \eqref{genN2} to the bounded function $g_M(n)=(n\wedge M)^2$ for $M\in \N$, we get
	\begin{align*}
	\bar{\Lcal}^L_{\eta(t)}  (n\wedge M)^2&=\bar{C}n\left(\left((n+1)\wedge M\right)^2-(n\wedge M)^2\right) \\ 
 & \quad +  2C\sum_{k\geq 0}(1+k)F^L_{k} (\eta(t))  \left( \left((2n+k)\wedge M\right)^2-(n\wedge M)^2 \right) \nonumber  \\
 % & = \bar{C}(n\wedge M)\left(\left((n+1)\wedge M\right)^2-(n\wedge M)^2\right)\\ 
 % & \quad +  2C\sum_{k\geq 0}(1+k)F^L_{k} (\eta(t))  \left( \left((2n+k)\wedge M\right)^2-(n\wedge M)^2 \right) \nonumber  \\
  & \leq 
 \bar{C}(n\wedge M)\left(2(n\wedge M)+1\right) \\ & \quad+  2C\sum_{k\geq 0}(1+k)F^L_{k} (\eta(t))  \left(3(n\wedge M)^2+k^2+4k(n\wedge M) \right)\nonumber  \\
 & \leq 6\bar C (n\wedge M)^2+2C \sum \limits_{k\geq 0} (1+k)^3F^L_k(\eta(t))+8C (n\wedge M)\sum_{k\geq 0}(1+k)^2F^L_{k} (\eta(t))
 \end{align*}
 where we used that for $n>M$, $\left((n+1)\wedge M\right)^2-(n\wedge M)^2=0.$ 
Conditional on $\eta [0,T]$ for some arbitrary $T>0$, $\big(\bar W^L (t):t\in [0,T]\big)$ is a Markov process with time-dependent generator $\bar{\Lcal}^L_{\eta(t)}$. 
Therefore, applying Dynkin's formula (see e.g. Appendix 1.5, Lemma 5.1 \cite{kipnis2013scaling}) and taking expectation over $\eta [0,T]$ we get
\begin{align}\label{eq:dynkin}
 \frac{d}{dt}\bar m^L_{2,M}(t) &=\E\big[ \bar{\Lcal}^L_{\eta(t)} \left((\bar{\W}^L(t))^2\wedge M\right) \big] \nonumber \\
 & \leq D \bigg(\bar{m}^L_{2, M} (t) + {m}^L_3 (t) + \E\bigl[(\bar{\W}^L(t)\wedge M) \sum_{k\geq 0}k^2F^L_{k} (\eta(t)) \bigr] \bigg)
 \end{align}
 where $D>0$ is some absolute constant (independent of $L$ and $M$). Regarding the last term, we have from Cauchy-Schwarz inequality (and since $\bar{m}^L_{2,M}(t)\geq 1$)
\begin{equation*}
    \E\Big[(\bar{\W}^L(t) \wedge M)\sum_{k\geq 0}k^2F^L_{k} (\eta(t)) \Big]\leq  
    \bar{m}^L_{2,M}(t) \Biggl( \E\Biggl[ \Big(\sum_{k\geq 0}k^2F^L_{k} (\eta(t))\Big)^2 \Biggr] \Bigg)^{1/2} .
\end{equation*}
Using that $\{\frac{kF^L_k(\eta(t))}{\sfrac{N}{L}} \}_{k\in \N}$ is a probability mass function, Jensen's inequality implies
\begin{align*}
    \E\Biggl[ \Big(\sum_{k\geq 0}k^2F^L_{k} (\eta(t))\Big)^2 \Biggr] &=\left(\frac{N}{L}\right)^2\E\Biggl[ \Big(\sum_{k\geq 0}k \frac{kF^L_{k} (\eta(t))}{\sfrac{N}{L}}\Big)^2 \Biggr] \nonumber\\&\leq \left(\frac{N}{L}\right)^2\E\Biggl[ \sum_{k\geq 0}k^2 \frac{kF^L_{k} (\eta(t))}{\sfrac{N}{L}}\Biggr] \leq\rho m^L_3(t)\ .
\end{align*}
Therefore, based on Proposition \ref{gwall}, we find 
 \begin{equation*}
   \frac{d}{dt} \bar{m}^L_{2,M} (t) \leq D\left(\bar{m}^L_{2,M} (t) +B_3e^{B_3t} +\sqrt{\rho B_3}e^{B_3t/2} \bar{m}^L_{2,M} (t) \right) \leq  D_2e^{D_2t} \bar{m}^L_{2,M} (t) 
 \end{equation*}
for another absolute constant $D_2 >0$. Since $T>0$ was arbitrary and based on Gronwall's inequality and conditions \eqref{init} and \eqref{initialcon0d}, we have
 \begin{equation*}
    \E \big[(\bar W^L(t)\wedge M)^2 \big] \leq   e^{e^{D_2t}-1} \bar{m}^L_{2,M} (0) \leq \alpha_4 e^{e^{D_2t}} \quad\text{for all } t\geq 0, L\geq 2, M\in\N \ .
 \end{equation*}
 Taking $M\to \infty$, the result then follows by monotone convergence.
 \end{proof}
% Therefore, taking $L\to \infty$ and by monotone convergence theorem and based on Proposition \ref{gwall}, we find:
%  \begin{align*}
%   \E \big[ (\bar W^L(t))^2 \big]-\E \big[ (\bar W^L(0))^2 \big]   &\leq D \int \limits_0^t\left(\bar{m}^L_2 (s) + {m}^L_3 (s) + \sqrt{\rho{m}^L_3 (s)}  \bar{m}^L_2 (s) \right)ds\\
%   &\leq D \int \limits_0^t\left(\bar{m}^L_{2} (s) +B_3e^{B_3s} +\sqrt{\rho B_3}e^{B_3s/2} \bar{m}^L_{2} (s) \right)ds \\&\leq \int \limits_0^t D_2e^{D_2s} \bar{m}^L_{2} (s)ds 
%  \end{align*}
%  for some absolute constants $D, \; D_2>0.$
%  Based on condition \eqref{initialcon0d} and \eqref{init}, we have
%  \begin{equation*}
%     \bar{m}^L_{2} (t) \leq \alpha_4+\int \limits_0^t D_2e^{D_2s} \bar{m}^L_{2} (s)ds 
%  \end{equation*}
%  Then the conclusion follows by integral form of Gronwall's inequality.
% \end{proof}
\noindent
Therefore, \eqref{markovbarn} and Lemma \ref{mombarN} imply \quad
$
 \sup \limits_{L\geq 2}\P( \bar{\W}^L(T)>a) \leq  \frac{\alpha_4e^{e^{D_2 T}}}{a}\ ,
$
which vanishes as $a\to \infty$ and concludes the proof of (i).\\
\textbf{Proof of (ii):} It suffices to prove that for each $T>0:$
\begin{equation}\label{toshowii}
    \delta^{-1}\limsup \limits_{L\to \infty}  \sup \limits_{0\leq s \leq T} 
    % \P ( \text{at least two } \bar{\W}^L-\text{jumps  in } [s,s+\delta)) 
\E \Big[ \P_{\eta[0,T]} \left( \text{at least two } \bar{\W}^L-\text{jumps  in } [s,s+\delta) \right) \Big]    
    \to 0
    % \text{ as } \delta\to 0^+ \ ,
\end{equation}
%     and we write
% \begin{equation}
%     \P \left( \text{at least two } \bar{\W}^L-\text{jumps  in } [s,s+\delta) \right)=\E \Bigg[ \P_{\eta[0,T]} \left( \text{at least two } \bar{\W}^L-\text{jumps  in } [s,s+\delta) \right) \Bigg]
% \end{equation}
as $\delta\to 0^+$, where $\P_{\eta[0,T]}$ is the conditional measure given the path of $\eta$ on $[0,T]$. Since $\bar W^L$ is a Markov jump process with finite rates, the probability of two or more jumps is of order $\delta^2$ which implies (ii). However, we do not have uniform-in-$L$ control on rates of $\bar W^L$, which requires a slightly technical analysis presented in Appendix \ref{app:ii}. 

\noindent
\textbf{Proof of (iii):} 
%  Since the (positive)  process $\bar{\W}^L(t)$ is increasing (as a function of $t$), it suffices to prove that for each $\epsilon>0,$
% $$\limsup \limits_{L\to \infty} \P \left(  |\bar{\W}^L(t) -\bar{\W}^L(0)|>\epsilon \right) \to 0 \text{ as } t\to 0^+ \  .$$
By Markov's inequality, we have 
\begin{equation}\label{markoviii}
   \P \left(  \bar{\W}^L(t) -\bar{\W}^L(0)>\epsilon \right)\leq  \frac{\E \big[ \bar{\W}^L(t) -\bar{\W}^L(0) \big] }{\epsilon}\ .
\end{equation}
% \textcolor{red}{Since we have no absolute values we can use Dynkin's formular directly and don't have to consider the quadratic variation, right? (expectation of the martingale vanishes)
% \[
% \E \big[ \bar{\W}^L(t) -\bar{\W}^L(0) \big] =\int_0^t \E\big[ \bar{\Lcal}^L_{\eta(s)}\bar{\W}^L(s)\big]\, ds
% \]
% }
% Conditional on $\eta [0,T]$ for arbitrary fixed $T>0$, $\big(\bar W^L (t):t\in [0,T]\big)$ is a Markov process with time-dependent generator $\bar{\Lcal}^L_{\eta(s)}$.
Using the same reasoning as in \eqref{eq:dynkin}, we get with Dynkin's formula 
% By It\^{o}'s formula for the joint process $\big( (\eta (t) ,\bar W^L (t)):t\geq 0\big)$, which is Markov, using the simple test function $G(\eta ,w)=w$ we have for all $t>0$
% 	\begin{equation}
% 		\label{eq:itoeta}
% 		\bar{\W}^L(t)-\bar{\W}^L(0)=\int_{0}^{t} \bar{\Lcal}^L_{\eta(s)}\bar{\W}^L(s)\, ds +M^L(t)\ ,
% 	\end{equation}
% where $(M^L (t) : t>0)$ is a martingale (see also \cite{kipnis2013scaling}, Lemma 5.1 in Appendix 1). 
 %with predictable quadratic variation given by integrating the 'carr\'e du champ' operator {(see e.g. \cite{kipnis2013scaling}, Appendix 1.5)}
	% \begin{equation}
	% 	\label{eq:mart}
	% 	[M^L](t)=\int_0^t \Big(\bar{\Lcal}^L_{\eta(s)}\left(\bar{\W}^L(s)\right)^2-2\bar{\W}^L(s)\bar{\Lcal}^L_{\eta(s)}\bar{\W}^L(s)\Big) ds\ .
	% \end{equation}
	% Based on \eqref{markoviii}, we have to bound
	\begin{equation}\label{tobound}
		\E \big[ \bar{\W}^L(t)\wedge M -\bar{\W}^L(0)\wedge M \big] = \int_0^t \E \big[\bar{\Lcal}^L_{\eta(s)}\left(\bar{\W}^L(s)\wedge M\right)\big] ds\ .
	\end{equation}
	% where we used H\"older's inequality and $\E\big[ (M^L(t))^2\big] =\E\big[ [M^L](t)\big]$ for the martingale.
Based on \eqref{genN2}, we have	
	% Regarding the first term on the right of \eqref{tobound}, we have {with \eqref{gen\W2}}
\begin{align*}
	0\leq \bar{\Lcal}^L_{\eta(s)} \left(\bar{\W}^L(s)\wedge M\right) & =\bar{C}\bar{\W}^L(s) \left((\bar{\W}^L(s)+1)\wedge M-\bar{\W}^L(s)\wedge M\right) \\ & \quad
 +2C\sum_{k\geq 0}(1+k)F^L_{k} (\eta(s)) \left((2\bar{\W}^L(s)+k)\wedge M -\bar{\W}^L(s)\wedge M \right) \nonumber \\
 &\leq \bar{C} (\bar{\W}^L(s)\wedge M)+ 2C\sum_{k\geq 0}(1+k)^2F^L_{k} (\eta(s))  \ .
\end{align*}
Therefore,
\begin{equation}\label{tight1}
    0\leq \E \big[\bar{\Lcal}^L_{\eta(s)}(\bar{\W}^L(s)\wedge M)\big] \leq D(\bar{m}^L_{2,M}(s)+{m}^L_2(s))
\end{equation}	
for some absolute constant $D>0$ (independent of $L$ and $M$).
Thus, 
% taking expectations on both sides and 
taking $M\to \infty$, by the monotone convergence theorem, Proposition \ref{gwall} and Lemma \ref{mombarN}, we conclude with \eqref{tobound} that
	\begin{equation}
		0\leq \E \big[ \bar{\W}^L(t) -\bar{\W}^L(0) \big]\leq D(\alpha_4e^{e^{D_2t}}+B_2e^{B_2t})t\to 0
	\end{equation}
	as $t \to 0$, which holds uniformly in  $L \geq {2}$ and concludes the proof of condition (iii).
 
 % Similarly, based on estimate  \eqref{tight1} and Proposition \ref{gwall} and Lemma  \ref{mombarN}, condition (iv) is also proved.
\end{proof}
By Prokhorov’s theorem, the tightness result in Proposition \ref{tight\W} implies the existence of sub-sequential limit points of the sequence $(\W^L(t) : t \geq 0)$ in the usual Skorohod topology of weak convergence on path space $D_{[0, \infty )} (\N )$ (see e.g.\ \cite{billingsley2013convergence}, Section 16). We denote the law of any such limit by $\mathbb{Q}$.
% More specifically, we have  existence of sub-sequential weak limits of $\mathbb{Q}^L$ in the Skorohod topology, and we denote any such limit by $\mathbb{Q}$.

\subsection{Characterisation of the limit process}

\noindent
In order to identify the limit $\mathbb{Q}$ we need to show that for all $t  \geq 0$ and $g \in C_b (\N )$, 
\begin{equation}\label{eq:marti}
	g(\omega(t))-g(\omega (0))-\int_0^t \hat{\Lcal}_s g(\omega(s))ds \; \text{ is a martingale w.r.t. }\mathbb{Q}\ ,
\end{equation}
where $\omega\in D_{[0,  \infty)} (\N)$ denotes an element in path space.  Together with the uniqueness of the martingale problem associated with $\hat{\Lcal}_t$, this implies convergence of $\mathbb{Q}^L$ and characterizes the limit $\mathbb{Q}$ as the law of the Markov process $(\hat{\W}(t): \; t\geq 0)$  with generator $\hat{\Lcal}_t$ \eqref{limgen}. 
Following a standard argument presented in Appendix \ref{sec:michalis}, we only need to prove that for all $  t\geq 0$
\begin{equation}\label{toshowa}
	\E^L \left[ \left| \int_0^{  {t}} \left( \hat{\Lcal}_{  {s}} g(\W^L({  {s}})) -\hat{\Lcal}^L_{\eta({  {s}})} g(\W^L({  {s}}))\right)d{  {s}} \right| \right]  \to 0\quad\mbox{as }
L\to\infty\ .
\end{equation}
% as $L \to \infty.$
% \begin{proof}[Proof of \eqref{toshow}]
Since the process $t\mapsto\hat{\Lcal}_{  {\eta(t)}}^L g(\W^L(t))$ is bounded in $L^1$-norm on compact time intervals uniformly with respect to $L$, and using the triangle inequality it suffices to prove that
% \textcolor{red}{?? CHECK}

% \textcolor{red}{
% \begin{align*}
%  &\E^L \left[ \left| \int_0^{  {t}} \left( \hat{\Lcal}_{  {s}} g(\W^L({  {s}})) -\hat{\Lcal}^L_{\eta({  {s}})} g(\W^L({  {s}}))\right)d{  {s}} \right| \right]\leq  \int_0^{t} \E^L \left[\left| \hat{\Lcal}_{s} g(\W^L({s})) -  \frac{L-1}{L} \hat{\Lcal}_{\eta({s})} g(\W^L({s}))\right|\right] d{s} \\ &\quad+ 
%  \frac{1}{L}\int_0^{t} \E^L \left[\left|  \hat{\Lcal}_{\eta({s})} g(\W^L({s}))\right|\right] d{s} 
% \end{align*}
% }

% \textcolor{red}{
% \begin{align*}
%  \left|  \hat{\Lcal}_{\eta({t})} g(n))\right| &\leq 4C\|g\|_{\infty} \sum_{k \geq 1} kF^L_{k}(\eta(t)) (1+n) \\& {+}2C\|g\|_{\infty}\bigg( n\sum_{k\geq 0} (1+k)F^L_{k}(\eta(t))+\sum_{k\geq 0} (1+k) F^L_{k}(\eta(t))  \bigg)\\& +2C\|g\|_{\infty}\frac{n}{L-1}(1+n) \\& \leq D\|g\|_{\infty}(1+\rho)(1+n)
% \end{align*}
% }
% \textcolor{red}{
% Therefore, 
% \begin{equation*}
%     \E^L \left[\left|  \hat{\Lcal}_{\eta({t})} g(\W^L({t}))\right|\right]\leq D\|g\|_{\infty}(1+\rho)(1+\hat m^L_1(t))
% \end{equation*}
% }
	\begin{equation}\label{conv_to_zero}
		\int_0^{t} \E^L \left[\left| \hat{\Lcal}_{s} g(\W^L({s})) -  \frac{L-1}{L} \hat{\Lcal}_{\eta({s})}^L g(\W^L({s}))\right|\right] d{s} \ \to 0
	\end{equation}
	as $L\to \infty$. 
	% \begin{multline}\label{diff_gen}
	% 	\left|  \hat{\Lcal}_t g(\W^L(t)) -  \frac{L-1}{L} \hat{\Lcal}_{\eta(t)} g(\W^L(t)) \right|= \Bigg| \sum_{k \geq 1} c(k,\W^L(t))\left(F^L_{k}(\eta(t))-f_{k}(t)\right) \big(g(\W^L(t)+1)-g(\W^L(t))\big)\\ +\frac{1}{\W^L(t)}\sum_{k\geq 0} c(\W^L(t),k) \left(F^L_{k}(\eta(t))-f_k(t) \right)\left( g(k+1)-g(\W^L(t)) \right)\\+ \frac{\W^L(t)-1}{\W^L(t)} \sum_{k\geq 0} c(\W^L(t),k)\left(F^L_{k}(\eta(t))-f_k(t)\right)\big( g(\W^L(t)-1)-g(\W^L(t))\big) \\-\frac{1}{L}c(\W^L(t),\W^L(t))\left[\frac{\W^L(t)+1}{\W^L(t)} \left( g(\W^L(t)+1)-g(\W^L(t)) \right)+ \frac{\W^L(t)-1}{\W^L(t)} \big( g(\W^L(t)-1)-g(\W^L(t))\big)\right] \Bigg|
	% \end{multline}
	Since $g\in C_b (\N )$ 
	and because of condition \eqref{eq:lip}, we find
	\begin{align*}
		&\left| \hat{\Lcal}_{s} g(\W^L{(s)}) -  \frac{L-1}{L} \hat{\Lcal}_{\eta{(s)}}^L g(\W^L{(s)}) \right| \leq 2||g||_{\infty} \Bigg( \sum_{k \geq 1} c(k,\W^L{(s)})\left|F^L_{k}(\eta{(s)})-f_{k}{(s)}\right| \\
		 &\qquad +\sum_{k\geq 0} c(\W^L{(s)},k)\left|F^L_{k}(\eta{(s)})-f_k{(s)}\right|+\frac{2C (\W^L(s))^2}{L}\Bigg)\\
		% \left|  \hat{\Lcal}_t g(\W^L{(s)}) -  \frac{L-1}{L} \hat{\Lcal}_{\eta{(s)}} g(\W^L{(s)}) \right|
		&\leq 2||g||_{\infty} \Bigg(   4C\W^L(s) \sum_{k \geq 1} k\left|F^L_{k}(\eta{(s)})-f_{k}{(s)}\right|
		+C \W^L(s)\left|F^L_{0}(\eta(s))-f_0{(s)} \right|+\frac{2C (\W^L (s))^2}{L}\Bigg)  {\ .}
	\end{align*}
	Notice that for all $M>0,$   {$s\leq t,$} we have
	\begin{align*}
		&\E^L \bigg[ \W^L{(s)}\sum_{k \geq 1} k\left|F^L_{k}(\eta{(s)})-f_{k}{(s)}\right| \bigg] \\
  & \quad = \E^L \bigg[ \W^L{(s)}\sum_{k \geq 1} k\left|F^L_{k}(\eta{(s)})-f_{k}{(s)}\right| \Big( \mathbbm{1}  \{ \W^L{(s)}\leq M \} +\mathbbm{1} \{ \W^L(s)>M \}\Big)\bigg]  \\ 
  % &\quad+ \E^L \bigg[ \W^L{(s)}\sum_{k \geq 1} k\left|F^L_{k}(\eta{(s)})-f_{k}{(s)}\right| \mathbbm{1} \{ \W^L(s)>M \}\bigg]\\
        & \quad \leq M\E^L \bigg[\sum_{k \geq 1} k\left|F^L_{k}(\eta{(s)})-f_{k}{(s)}\right|  \bigg] +  2\rho\sup_{L\geq 2,  s \leq t}\E^L \big[ \W^L(s)\mathbbm{1} \{ \W^L(s)>M\} \big]\ .
	\end{align*}
An analogous estimate holds for $\E^L \Big[ \W^L(s) \left|F^L_{0}(\eta (s))-f_{0}(s)\right| \Big]$ 
 % Similarly,
 % \begin{align*}
 % \E^L \Big[ \W^L(s) \left|F^L_{0}(\eta (s))-f_{0}(s)\right| \Big] &\leq M\E^L \Big[ \left|F^L_{0}(\eta(s))-f_{0}(s)\right|  \Big] \\&\quad+ 2\sup\limits_{L\geq 2,s \leq t}\E^L \big[ \W^L(s) \mathbbm{1} \{ \W^L(s)>M\} \big]
 % \end{align*}
 %   {Similarly, for all $M>0,\; s\leq T,$ we have
 % \begin{multline*}
 %     \E^L \bigg[ \W(s)\sum_{k \geq 0} \left|F^L_{k}(s)-f_{k}(s)\right| \bigg] \leq M\E^L \bigg[\sum_{k \geq 0} \left|F^L_{k}(\eta{  {(s)}})-f_{k}{  {(s)}}\right|  \bigg] \\+ 2\sup_{L\in \N,{  {s}} \leq T}\E^L \big[ \W(s) \mathbbm{1} \{ \W(s)>M\} \big]\ 
 % \end{multline*}}
	% Therefore, 
and with \eqref{sec_mom_boundN} we find:
	\begin{multline*}
		\int_0^{t} \E^L \left[\left| \hat{\Lcal}_s g(\W^L(s)) -  \frac{L-1}{L} \hat{\Lcal}_{\eta(s)} g(\W^L(s))\right|\right] ds  \\   \leq  2||g||_{\infty}   \Bigg( 4CM  \int_0^t \E^L \bigg[\sum_{k \geq 1} k\left|F^L_{k}(\eta(s))-f_{k}(s)\right|  \bigg] ds  \\   + 2t  C \left( 1+4\rho \right) \sup_{L\geq 2,s \leq t}\E^L \big[ \W^L(s) \mathbbm{1} \{ \W^L(s)>M\} \big] 
  \\ 
  +CM  \int_0^t \E^L \bigg[  \left|F^L_{  {0}}(\eta(s))-f_{  {0}}(s) \right| \bigg] ds+\frac{2C  { \big(\alpha_4+C_2 t\big) e^{C_2 t}} }{L}t \Bigg)   {\ .}  %\big(A_2 {+}B_2T{+}C_2 \alpha_3\big) e^{C_2 T} }{L}T  {\ .}
	\end{multline*}
	In the limit $L\to \infty$, based on Theorem \ref{thmfactorization}, we have that
	$$
	\int_0^t \E^L \bigg[\sum_{k \geq 1} k\left|F^L_{k}(\eta(s))-f_{k}(s)\right|  \bigg] ds \to 0\quad\mbox{and}\quad \int_0^t \E^L \bigg[\left|F^L_{0}(\eta(s))-f_{0}(s)\right|  \bigg] ds \to 0   {\ .}
	$$
	Therefore, for all $M>0,$
	\begin{multline*}
		  {\limsup \limits_{L\to \infty}}\int_0^t \E^L \left[\left| \hat{\Lcal}_s g(\W^L(s)) -  \frac{L-1}{L} \hat{\Lcal}^L_{\eta(s)} g(\W^L(s))\right|\right] ds   \\   \leq  4\|g\|_{\infty}t    C \left( 1+4\rho \right) \sup_{L\geq 2,s\leq t}  \E^L \big[ \W^L(s) \mathbbm{1} \{ \W^L(s)>M\} \big]\ .
	\end{multline*}
	In the limit $M\to \infty$, the uniform integrability of $\{\W^L(s)\}_{L\geq 2,s\leq t}$ due to relation \eqref{sec_mom_boundN}, gives \eqref{conv_to_zero}.
% \end{proof}

\section*{Acknowledgements}
The authors are grateful to Michail Loulakis for useful discussions, and to the anonymous referee whose detailed comments helped to improve the paper.

\bibliographystyle{amsplain}
\bibliography{ref_new}

\appendix

\clearpage

\section*{Appendix}

\section{Proof of Theorem \ref{thmfactorization}\label{app:a}}
Here, we present a modification of the proof of Proposition 1 in \cite{grosskinsky2019derivation}, where tightness of the process $\Big(\sum_{k\geq 0} F_k^L (\eta(t) )\, h(k) :t  \geq 0\Big)$ was established for bounded functions $h: \N_0 \to \R$. In our proof, tightness can be established also for Lipschitz functions $h$ and in particular without any assumption on the initial conditions  as stated below.
\begin{proposition}
\label{propexistence}
Consider a process with generator \eqref{eq:GenMis} on the complete graph with sublinear rates (\ref{eq:lip}). For any Lipschitz function $h$, denote by $\mathbb{Q}_h^L$ the measure of the process $t \mapsto H(\eta (t)):=\big\langle F^L (\eta (t)),h\big\rangle$ on path space $D_{[0,\infty)}(\R )$, which is the image measure of $\mathbb{P}^L$ under the mapping $\eta\mapsto \big\langle F^L (\eta),h\big\rangle$. Then $\mathbb{Q}_h^L$ is tight as $L \to \infty$.
\end{proposition}
\begin{proof}
Using a version of Aldous' criterion to establish tightness for $\mathbb{Q}_h^L$ (cf. Theorem 16.10 in \cite{billingsley2013convergence}), it suffices to show that for all $t\geq 0$
	\begin{equation}
		\label{eq:tight1}
		\lim_{a\to \infty}\limsup_{L\to\infty} \mathbb{P}^L \big[ |H(\eta (t))|\geq a\big] =0,
	\end{equation}
	 and that for any $\epsilon>0$, $t>0$,
	\begin{equation}
		\label{eq:tightness}
		\lim_{\delta_0\to 0^+}\limsup_{L\to \infty}\sup_{\delta\leq \delta_0}\sup_{\tau \in \mathfrak{T}_t} {\mathbb{P}}^L\big[ |H(\eta(\tau+\delta))-H(\eta(\tau))|>\epsilon\big] =0,
	\end{equation}
 where $\mathfrak{T}_t$ is the set of stopping times satisfying $\tau\leq t.$\\
%, while in \eqref{eq:tight1} the initial condition is given by i.i.d random variables with asymptotic density $\rho$ as given in \eqref{initialcon1}.
Since $h$ is Lipschitz, $|h(k)|\leq |h(0)|+\|h\|_{\text{Lip}}k$ for all $k\in \N_0$ and  $$\big|\big\langle F^L (\eta),h\big\rangle\big|\leq |h(0)|+\|h\|_{\text{Lip}}\rho$$ is uniformly bounded in $L$ and $\eta\in E_{L,N}$, \eqref{eq:tight1} follows easily from Markov's inequality,
\[
\mathbb{P}^L \big[ |H(\eta (t))|\geq a\big]\leq \frac{\mathbb{E}^L \big[ |H(\eta (t))|\big]}{a}\leq \frac{|h(0)|+\|h\|_{\text{Lip}}\rho}{a}\quad\mbox{for all }L\geq 2\ .
\]
Now fix $\delta_0> 0$, $\tau \in \mathfrak{T}_t$ and consider $\delta<\delta_0$.  By It\^{o}'s formula, we have for all $u>0$
	\begin{equation}
		\label{eq:itoetaMIM}
  H({\eta}(u+\delta))-H(\eta(u) )=\int_u^{u+\delta} \Lcal H({\eta} (s))\, ds +M_h (u+\delta)-M_h(u)\ ,
	\end{equation}
	where $(M_h (u) : u\geq 0)$ is a martingale with predictable quadratic variation given by integrating the 'carr\'e du champ' operator  
\begin{equation}
\label{eq:mart}
[M_h ](t)=\int_0^t \big[\mathcal{L}H^2-2H\mathcal{L}H\big] ({\eta}(s))ds\ .
\end{equation}
To compute $\Lcal H(\eta )$, we first recall that 
\begin{equation}
% \mathcal{L}F_k({\eta})
% =\frac{1}{L}\sum_{x\in \Lambda} \mathcal{L}\delta_{\eta_x,k} \,
H(\eta)=\langle h, F^L({\eta}) \rangle =\sum_{k\geq 0}h(k) \frac{1}{L}\sum_{x\in\Lambda} \delta_{\eta_x,k}= \frac{1}{L}\sum_{x\in\Lambda} h(\eta_x)
\end{equation}
Therefore,
\begin{equation}
\mathcal{L}H({\eta})=
% \frac{1}{L}\sum_{x\in\Lambda}  \sum \limits_{y\neq x} \frac{1}{L-1}\left(c(\eta_y,\eta_x)(h(\eta_x+1)-h(\eta_x)) + c(\eta_x,\eta_y)(h(\eta_x-1)-h(\eta_x)) \right) 
\frac{1}{L}\frac{1}{L-1}\sum_{x\in\Lambda}  \sum \limits_{y\neq x} c(\eta_x,\eta_y) \big[ \left(h(\eta_x-1)-h(\eta_x)\right) + \left(h(\eta_y+1)-h(\eta_y)\right)\big]
\label{compu}
\end{equation}
Thus, for all $\eta \in E_{L,N}$, we have
\begin{align*}
\big|\mathcal{L}H({\eta})\big|&\leq\frac{2C\|h\|_{\text{Lip}}}{L}\frac{1}{L-1}\sum_{x\in\Lambda}  \sum \limits_{y\neq x}  \eta_x (1+\eta_y)\\
&\leq 2C\|h\|_{\text{Lip}}\frac{N}{L}\frac{L+N}{L-1}\leq 4C\|h\|_{\text{Lip}} \rho (1+\rho).
\end{align*}
where  $N=N_L=\sum \limits_{x\in\Lambda} \eta_x$ is the (constant) number of particles. \\ % for which we have  $\frac{N}{L}\leq \alpha_1.$\\
Using again Markov's inequality in \eqref{eq:tightness} and replacing $u$ by the bounded stopping time $\tau\leq t,$ we have to bound
	\begin{align}\label{toboundMim}
		\E^L\Big[\big| H(\eta(\tau+\delta))-H(\eta(\tau))\big|\Big]&\leq \E^L \Big[\int \limits_{\tau}^{\tau+\delta}  |{\Lcal} H(\eta(s))|\,ds\Big] +  \E^L \big[ \left(M_h (\tau+\delta)-  M(\tau) \right)^2\big]^{1/2} \ \nonumber \\
        % &= \E^L \Bigg[\int_{\tau}^{\tau+\delta}  |{\Lcal} H(\eta(s))|\,ds\Bigg] +  \E^L \big[ M_h^2(\tau+\delta)-  M_h^2(\tau) \big]^{1/2} \ \nonumber\\
        &\leq \delta_0   \left( 4C\|h\|_{\text{Lip}} \rho (1+\rho)\right) + \E^L \big[  [M_h ](\tau+\delta)-  [M_h ](\tau)\big]^{1/2}
	\end{align}
	where we used H\"older's inequality and the stopping time theorem for the martingale $M_h^2 (t)-[M_h ](t)$. 
Then, to control the last term of \eqref{toboundMim}, it suffices to bound (uniformly in $L$ and $\eta$) the 'carr\'e du champ' operator, for which we have 
\begin{align*}
\big[\mathcal{L}H^2-2H\mathcal{L}H\big] ({\eta})&=\frac{1}{L^2}\frac{1}{L-1}\sum_{\overset{x,y\in\Lambda}{y\neq x}}  c(\eta_x,\eta_y) \big[ \left(h(\eta_x-1)-h(\eta_x)\right) + \left(h(\eta_y+1)-h(\eta_y)\right)\big]^2 \nonumber \\ 
&\leq\frac{2}{L^2}\frac{1}{L-1}\sum_{\overset{x,y\in\Lambda}{y\neq x}} c(\eta_x,\eta_y) \big[ \left(h(\eta_x-1)-h(\eta_x)\right)^2 + \left(h(\eta_y+1)-h(\eta_y)\right)^2 \big]\nonumber\\
&\leq 4C \|h\|^2_{\text{Lip}}\,\frac{1}{L}\frac{N}{L}\frac{L+N}{L-1}\leq \frac{8C\|h\|^2_{\text{Lip}}\rho(1+\rho)}{L}
% &\quad +2\left(h(\eta_x-1)-h(\eta_x)\right) \left(h(\eta_y+1)-h(\eta_y)\right)\bigg]
\end{align*}
% Therefore, we have
% \begin{equation}
%    \big| \mathcal{L}H^2(\eta)-2H(\eta)\mathcal{L}H(\eta) \big| \leq 4C \|h\|^2_{\text{Lip}}\,\frac{1}{L}\frac{N}{L}\frac{L+N}{L-1}\leq \frac{8C\|h\|^2_{\text{Lip}}\rho(1+\rho)}{L}
% \end{equation}
uniformly in $\eta \in E_{L,N}$. Therefore, 
\begin{equation}\label{mimprove}
\E^L \big[  [M_h ](\tau+\delta)-  [M_h ](\tau)\big]=\E^L \bigg[ \int_{\tau}^{\tau+\delta} \big[\mathcal{L}H^2-2H\mathcal{L}H\big] ({\eta}(s))ds \bigg]\leq \delta_0  \frac{8C\|h\|^2_{\text{Lip}}\rho(1+\rho)}{L}    
\end{equation}
which vanishes as $\delta_0\to0^+$, finishing the proof.
\end{proof}

Note that with \eqref{mimprove} the martingale $(M_h (u):u\geq 0)$ vanishes also as $L\to\infty$ on arbitrary compact time intervals, which implies that a generalized version also of the main result in \cite{grosskinsky2019derivation} holds as formulated in Theorem \ref{thmfactorization}.\\

% \textcolor{red}{The proof for tightness doesn't require any assumption regarding the initial moments, but we said that to derive the mean-field equation we needed the second moment bound to kill the diagonal terms (see (24) Mim's proof). So we should include this assumption in the statement of Theorem 2.1 (no of Prop. A.1), right?}

\section{Proof of \eqref{toshowii}}\label{app:ii}

By the law of total probability we have
\begin{align*}
    &\P_{\eta[0,T]} \left( \text{at least two } \bar{\W}^L-\text{jumps  in } [s,s+\delta) \right)\\
    &\quad=\sum \limits_{n=1}^{\infty} \P_{\eta[0,T]} \left( \text{at least two } \bar{\W}^L-\text{jumps  in } [s,s+\delta) \big| \bar{\W}^L(s)=n \right) \P_{\eta[0,T]} \left(\bar{\W}^L(s)=n \right)\ .
\end{align*}
% \begin{align*}
%     &\P_{\eta[0,T]} \left( \text{at least two } \bar{\W}^L-\text{jumps  in } [s,s+\delta) \big| \bar{\W}^L(s)=n \right)  \\&=1-\P_{\eta[0,T]} \left( \text{no } \bar{\W}^L-\text{jump  in } [s,s+\delta) \big| \bar{\W}^L(s)=n \right)\\&\quad-\P_{\eta[0,T]} \left( \text{one } \bar{\W}^L-\text{jump  in } [s,s+\delta) \big| \bar{\W}^L(s)=n \right)
% \end{align*}
We consider the following stopping times:
\begin{itemize}
    \item $\tau^L_1:=\inf\{t\geq s: \; \bar{\W}^L(t)>\bar{\W}^L(s)\}$, time of first jump of $\bar{\W}^L$ after $t=s$.
    \item $\tau^L_2:=\inf\{t\geq \tau_1: \; \bar{\W}^L(t)>\bar{\W}^L(\tau_1)\}$,
     time of second jump of $\bar{\W}^L$ after $t=s$.
\end{itemize}
Then the required probability is rewritten as:
\begin{equation*}
     \P_{\eta[0,T]} \left( \text{at least two } \bar{\W}^L-\text{jumps  in } [s,s+\delta) \big| \bar{\W}^L(s)=n \right) =\P_{\eta[0,T]} \left( \tau^L_2<s+\delta \big| \bar{\W}^L(s)=n \right)\ .
\end{equation*}
Therefore, we have
\begin{equation}\label{integral}
    \P_{\eta[0,T]} \left( \tau^L_2<s+\delta \big| \bar{\W}^L(s)=n \right)=\int \limits_{s}^{s+\delta} \P_{\eta[0,T]} \left( \tau^L_2<s+\delta \big| \bar{\W}^L(s)=n, \tau^L_1=t \right)f_n(t)dt\ ,
\end{equation}
where $f_n(t), \; t\geq s$ is the p.d.f of $\tau^L_1,$ (conditioned on $\eta_{ [0,T]}$ and $\bar{\W}^L(s)=n$) i.e. the p.d.f. of a (shifted) exponential random variable with rate equal to the total exit rate:
$$
r_n(t)=\bar{C}n+2C\sum \limits_{k\geq 0}(1+k)F^L_k(\eta(t))=\bar{C}n+ 2C(1+\sfrac{N}{L})\ , \quad t\geq s\ ,$$
which gives 
$$
f_n(t)=(\bar{C}n+ 2C(1+\sfrac{N}{L}))e^{-(\bar{C}n+ 2C(1+\sfrac{N}{L}))(t-s)}\ , \quad t\geq s\ .
$$
By the law of total probability, the right side of \eqref{integral} equals:
\begin{align*}
  % &\int \limits_{s}^{s+\delta} \P_{\eta[0,T]}\left( \tau^L_2<s+\delta \big| \bar{\W}^L(s)=n, \tau^L_1=t \right)f_n(t)dt  \\
  % &\quad = 
  &\sum \limits_{k\geq 0} \int \limits_{s}^{s{+}\delta} \P_{\eta[0,T]} \left( \tau^L_2 {<}s{+}\delta \big|  \tau^L_1 {=}t, \bar{\W}^L(t){=}2n{+}k \right) \P_{\eta[0,T]} \left( \bar{\W}^L(t){=}2n{+}k  \big| \bar{\W}^L(s){=}n, \tau^L_1 {=}t\right)  f_n(t)dt\\
  &+ \int \limits_{s}^{s+\delta} \P_{\eta[0,T]} \left( \tau^L_2{<}s{+}\delta \big|  \tau^L_1{=}t, \bar{\W}^L(t){=}n{+}1 \right) \P_{\eta[0,T]} \left( \bar{\W}^L(t){=}n{+}1  \big| \bar{\W}^L(s){=}n, \tau^L_1{=}t\right)  f_n(t)dt\ .
\end{align*}
Regarding the  terms with long jumps in the first line we have:
\begin{itemize}
    \item $\P_{\eta[0,T]} \left( \tau^L_2<s+\delta \big|  \tau_1=t, \bar{\W}^L(t)=2n+k \right)$\\
    Under the above conditional measure, $\tau^L_2$ follows a (shifted) exponential distribution with rate equal to the total exit rate: 
    \begin{equation*}
    r_{n+2k}(u)=\bar{C}(2n+k)+2C\sum \limits_{k\geq 0}(1+k)F^L_k(\eta(u))=\bar{C}(2n+k)+2C(1+\sfrac{N}{L}) 
    \end{equation*}
      for all $u\in [t, \infty)$.  Therefore,  for all $s\leq t \leq s+\delta,$
    \begin{align*}\P_{\eta[0,T]} \left( \tau^L_2<s+\delta \big|  \tau^L_1=t, \bar{\W}^L(t)=2n+k \right)&=1- e^{-\left(\bar{C}(2n+k)+2C(1+\sfrac{N}{L})  \right)(s+\delta-t)}\\
    &\leq \left( \bar{C}(2n+k)+2C(1+\sfrac{N}{L})  \right)(s+\delta-t) \\
    & \leq \bar{C}(2n+k+1) (s+\delta-t)
    \end{align*}
    \item $\displaystyle\P_{\eta[0,T]} \left( \bar{\W}^L(t)=2n+k  \big| \bar{\W}^L(s)=n, \tau^L_1=t\right)= \frac{2C(1+k)F^L_{k}(\eta(t))}{\bar{C}n+ 2C(1+\sfrac{N}{L}) }\ .$
\end{itemize}
In total, using that $\bar{C}:=2C(1+3\rho)>2C(1+\rho)$, we have 
\begin{align*}
    &\sum \limits_{k\geq 0} \int \limits_{s}^{s+\delta} \P_{\eta[0,T]} \left( \tau^L_2<s+\delta \big|  \tau_1=t, \bar{\W}^L(t)=2n+k \right) \P_{\eta[0,T]} \left( \bar{\W}^L(t)=2n+k  \big| \bar{\W}^L(s)=n, \tau^L_1=t\right)  f_n(t)dt\\
    &\quad \leq  \sum \limits_{k\geq 0} \int \limits_{s}^{s+\delta} \bar{C}(2n+k+1) (s+\delta-t) \frac{2C(1+k)F^L_{k}(\eta(t))}{\bar{C}n+ 2C(1+\sfrac{N}{L})} (\bar{C}n+ 2C(1+\sfrac{N}{L}))e^{-(\bar{C}n+ 2C(1+\sfrac{N}{L}))(t-s)}dt\\
     &\quad \leq  \sum \limits_{k\geq 0} \int \limits_{0}^{\delta}  \bar{C}(2n+k+1)(\delta-u) 2C(1+k)F^L_{k}(\eta(u+s))du \\
     &\quad \leq \delta  \bar{C}^2 \int \limits_{0}^{\delta}  \sum \limits_{k\geq 0} (1+k)^2F^L_{k}(\eta(u+s))du +2 \delta  \bar{C}n \int \limits_{0}^{\delta} 2C \sum \limits_{k\geq 0} (1+k)F^L_{k}(\eta(u+s))du \\
     &\quad \leq \delta  \bar{C}^2\int \limits_{0}^{\delta}  \sum \limits_{k\geq 0} (1+k)^2F^L_{k}(\eta(u+s))du + 2\delta^2  \bar{C}^2n  \ .
\end{align*}
Regarding the term with a one-step jump in the second line we have:
\begin{itemize}
    \item $\P_{\eta[0,T]} \left( \tau^L_2<s+\delta \big|  \tau^L_1=t, \bar{\W}^L(t)=n+1 \right)$\\
    Under the above conditional measure, $\tau^L_2$ follows a (shifted) exponential distribution with rate equal to the total exit rate: 
    \begin{equation*}
    r_{n+1}(u)=\bar{C}(n+1)+2C\sum \limits_{k\geq 0}(1+k)F^L_k(\eta(u))=\bar{C}(n+1)+2C(1+\sfrac{N}{L}) \ ,
    \end{equation*}
    for all $u\in [t, \infty)$. Therefore, 
    \begin{align*}
    \P_{\eta[0,T]} \left( \tau^L_2<s+\delta \big|  \tau^L_1=t, \bar{\W}^L(t)=n+1 \right)&=1- e^{-\left( \bar{C}(n+1)+2C(1+\sfrac{N}{L})\right)(s+\delta-t)}\\
    &\leq \left( \bar{C}(n+1)+2C(1+\sfrac{N}{L})\right)(s+\delta-t)\\
    & \leq \bar{C} (n+2)(s+\delta-t)
    \end{align*}
    \item $\displaystyle\P_{\eta[0,T]} \left( \bar{\W}^L(t)=n+1  \big| \bar{\W}^L(s)=n, \tau^L_1=t\right)= \frac{\bar{C}n}{\bar{C}n+ 2C(1+\sfrac{N}{L})} $  \ .
\end{itemize}
In total, we have: 
\begin{align*}
    &\int \limits_{s}^{s+\delta} \P_{\eta[0,T]} \left( \tau^L_2<s+\delta \big|  \tau^L_1=t, \bar{\W}^L(t)=n+1 \right) \P_{\eta[0,T]} \left( \bar{\W}^L(t)=n+1  \big| \bar{\W}^L(s)=n, \tau^L_1=t\right)  f_n(t)dt\\
    &\leq \int \limits_{s}^{s+\delta} \bar{C} (n+2)(s+\delta-t)\frac{\bar{C}n}{\bar{C}n+ 2C(1+\sfrac{N}{L})} \left(\bar{C}n+ 2C(1+\sfrac{N}{L})\right)e^{-(\bar{C}n+ 2C(1+\sfrac{N}{L}))(t-s)}dt\\
    &\leq  \int \limits_{0}^{\delta}  \bar{C} (n+2)(\delta-u)  \bar{C}ndu \leq \delta^2 \bar{C}^2 n(n+2)\ .
\end{align*}
Combining the above,
\begin{align*}
    &\P_{\eta[0,T]} \left( \text{at least two } \bar{\W}^L-\text{jumps  in } [s,s+\delta) \right)\\
    &\quad\leq \sum \limits_{n=1}^{\infty} \P_{\eta[0,T]} \left( \text{at least two } \bar{\W}^L-\text{jumps  in } [s,s+\delta) \big| \bar{\W}^L(s)=n \right) \P_{\eta[0,T]} \left(\bar{\W}^L(s)=n \right) \\
    &\quad = \delta  \bar{C}^2\int \limits_{0}^{\delta}  \sum \limits_{k\geq 0} (1+k)^2F^L_{k}(\eta(u+s))du + 2\delta^2  \bar{C}^2 \E_{\eta[0,T]} \bigl[\bar{\W}^L(s) \bigr] \\
    &\qquad + \delta^2 \bar{C}^2\E_{\eta[0,T]} \bigl[\bar{\W}^L(s)(\bar{\W}^L(s)+2) \bigr]  
\end{align*}
Therefore, for all $s\in [0,T]$, we have:
\begin{align*}
    \P&\left( \text{at least two } \bar{\W}^L-\text{jumps  in } [s,s+\delta) \right) \\ &\quad \leq \delta  \bar{C}^2\int \limits_{0}^{\delta} \E \Bigl[ \sum \limits_{k\geq 0} (1+k)^2F^L_{k}(\eta(u+s)) \Bigr]du + 2\delta^2  \bar{C}^2 \E \bigl[\bar{W}^L(s) \bigr] \\&\qquad+ \delta^2 \bar{C}^2 \E \bigl[\bar{\W}^L(s)(\bar{W}^L(s)+2) \bigr]  \\
    & \quad\leq \delta  \bar{C}^2\int \limits_{0}^{\delta} (1+2\rho+m^L_2(u+s))du + 5\delta^2  \bar{C}^2  \bar{m}^L_2(s)
\end{align*}
Based on Proposition \ref{gwall} and assumption \eqref{initialcon0c}, we have for all $u\in [0,\delta],\; s\in[0,T]$
$$
m^L_2(u+s)\leq B_2e^{B_2(u+s)}\leq B_2e^{B_2(T+\delta)}
$$
and from Lemma \ref{mombarN},
$$
\bar{m}^L_2(s)\leq\bar{m}^L_2(T)\leq \alpha_4e^{e^{D_2T}}\, .
$$
% Therefore,
% \begin{equation*}
%   \P\left( \text{at least two } \bar{N}^L-\text{jumps  in } [s,s+\delta) \right) \leq \delta^2  \bar{C}^2 \left(1+2\rho+B_2e^{B_2T}+5 \alpha_4e^{e^{D_2T}}\right) 
% \end{equation*}
Therefore,
\begin{multline*}
    \delta^{-1}\limsup \limits_{L\to \infty}  \sup \limits_{0\leq s \leq T} \P ( \text{at least two } \bar{W}^L-\text{jumps  in } [s,s+\delta))\\ \leq \delta  \bar{C}^2  \left( 1+2\rho+B_2e^{B_2(T+\delta)} +5\alpha_4e^{e^{D_2T}} \right) \to 0\quad\mbox{as }\delta\to 0\, .
\end{multline*}
% which goes to $0$ as $\delta\to 0.$ \\[2mm]

\section{Justification of \eqref{toshowa}}\label{sec:michalis}

Following \cite{armendariz2015metastability}, Section 8, in order to establish \eqref{eq:marti} we need to show that for any $T>0$ 
\begin{equation}\label{MG}
	\E^{\mathbb{Q}}\left[f\left((\omega (u): 0\leq u \leq s)\right) \left(g(\omega(t))-g(\omega (s))-\int_s^t \hat{\Lcal}_{  {u}} g(\omega(u))du\right)\right] =0
\end{equation}
for all $0\leq s\leq t \leq T$ and continuous bounded functions $f: D_{[0, T]} (\N ) \to \R$. Notice that since $T>0$ is arbitrary, this implies Theorem \ref{mainthm}. {Based on tightness estimates in the proof of Proposition \ref{tight\W},}  Lemma 8.1 in \cite{armendariz2015metastability} implies that, as $L\to \infty$, 
\begin{multline}
	\E^{\mathbb{Q}^L}\left[f\left((\omega (u): 0\leq u \leq s)\right) \left(g(\omega (t))-g(\omega (s))-\int_s^t \hat{\Lcal}_{  {u}} g(\omega (u))du\right)\right] \to \\
	\E^{\mathbb{Q}}\left[f\left((\omega (u): 0\leq u \leq s)\right) \left(g(\omega(t))-g(\omega (s))-\int_s^t \hat{\Lcal}_{  {u}} g(\omega(u))du\right)\right]\ .
\end{multline}
Therefore, in order to prove \eqref{MG}, it suffices to prove that 
% \begin{equation}
% 	\E^{\mathbb{Q}^L}\left[ \left|g(\omega (t))-g(\omega (s))-\int_s^t \hat{\Lcal}_{  {u}} g(\omega (u))du\right|\right] \to 0\quad  {\mbox{as }L\to\infty}
% \end{equation}
% which is equivalent to proving that
\begin{equation}
	\E^{L}\left[ \left|g(\W^L(t))-g(\W^L (s))-\int_s^t \hat{\Lcal}_{  {u}} g(\W^L(u))du\right|\right] \to 0\quad  {\mbox{as }L\to\infty}\ ,
\end{equation}
since $\mathbb{Q}^L$ is the law of the process $\big(W^L (t): t\geq 0 \Big)$. 
Since $\Big((\eta(t),W^L (t)): t\geq 0 \Big)$ is a Markov process, we know that 
% with initial condition $(\zeta ,x)$, 
the process
\begin{multline*}
	g(\W^L(t))-g(\W^L(0))-\int_0^t \hat{\Lcal}^L_{\eta(s)} g(\W^L(s))\, ds \\
 =g(\W^L(t))-g(\W^L(0))-\int_0^t \hat{\Lcal}_s g(\W^L(s))+ \int_0^t \left( \hat{\Lcal}_s g(\W^L(s)) -\hat{\Lcal}^L_{\eta(s)} g(\W^L(s))\right)ds
\end{multline*}
is a $\mathbb{P}^L$-martingale for all $L\geq 2$. Therefore it suffices to prove \eqref{toshowa}.

% \clearpage
\end{document}